\documentclass[11pt]{article}
\usepackage{natbib} 
\usepackage[margin=1cm]{caption}
\usepackage{array, amsmath,amsthm,amsfonts,amssymb,mathrsfs,epsfig,bm}
\usepackage{lipsum}
\usepackage[usenames]{color}
\usepackage[colorlinks=true]{hyperref}
\usepackage{graphicx}
\usepackage[scale=.8]{geometry}
\usepackage{ulem}
\usepackage{textcomp}
\usepackage{amsmath}
\usepackage{esint}
\usepackage{cleveref}
\usepackage[export]{adjustbox} % For images align
\usepackage{caption}
\usepackage{subcaption}
\usepackage{authblk}

\newtheorem{theorem}{Theorem}%[section]

\newtheorem{lemma}{Lemma}

\newtheorem{corollary}{Corollary}
\newtheorem{proposition}{Proposition}
\newtheorem{definition}{Definition}

\theoremstyle{remark}

\newtheorem{example}{Example}
\newtheorem{hypothesis}{Hypothesis}

 \def\beqlb{\begin{eqnarray}}\def\eeqlb{\end{eqnarray}}
 \def\beqnn{\begin{eqnarray*}}\def\eeqnn{\end{eqnarray*}}

\def\N{\mathbb{N}}
\def\Z{\mathbb{Z}}

\def\R{\mathbb{R}}

\def\E{\mathbb{E}}

\renewcommand{\epsilon}{\varepsilon}

\definecolor{mygray}{gray}{0.9}
\definecolor{deeppink}{RGB}{255,20,147}
\definecolor{mygreen}{rgb}{0.05, 0.576, 0.03}
\definecolor{myred}{rgb}{0.768, 0.09, 0.09}

\long\def\symbolfootnote[#1]#2{\begingroup
\def\thefootnote{\fnsymbol{footnote}}\footnote[#1]{#2}\endgroup}
\usepackage{MnSymbol}

\usepackage{authblk}
\newcommand{\ddr}{\mathrm{d}}

\begin{document}

\title{\bf  On a Markov chain related to the individual lengths in the recursive construction of Kingman's coalescent}
%A Markov chain in a non-Markov construction \\ 
%of Kingman's coalescent}
\author{Linglong Yuan
\thanks{Department of Mathematical Sciences,
University of Liverpool, Liverpool, UK\\
 \hspace*{1.5em} Email: \href{linglong.yuan@liverpool.ac.uk}{linglong.yuan@liverpool.ac.uk}}}

\maketitle
\abstract{Kingman's coalescent is a widely used process to model sample genealogies in population genetics. Recently there have been studies on the inference of quantities related to the genealogy of additional individuals given a known sample. This paper explores the recursive (or sequential) construction which is a natural way of enlarging the sample size by adding individuals one after another to the sample genealogy via individual lineages to construct the Kingman's coalescent. Although the process of successively added lineage lengths is not Markovian,  we show that it contains a Markov chain which records the information of the successive largest lineage lengths and we prove a limit theorem for this Markov chain. }

	\vspace{8pt} \noindent {\textit{Key words:} Kingman's coalescent, recursive construction, sequential construction, lineage length, provisional external branch length, convergence of non-Markov processes}
	\\
 

\section{Introduction}

\section{Introduction}

The coalescent theory was introduced by Kingman (\citealp{kingman1982coalescent}) and has since then become a standard framework to model sample genealogies.  The Kingman's $n$-coalescent with $n\geq 1$, denoted by $\Pi^n=(\Pi^{n}(t))_{t\geq 0},$ is a continuous-time Markov process with state space $\mathcal P(n)$, the set of partitions of $[n]:=\{1,2,\cdots,n\}$. It starts at time $0$ with the partition of singletons $\{\{1\},\{2\},\cdots, \{n\}\}$, and at any time, any two blocks merge into one at rate $1$ independently.  Eventually, the coalescent reaches the final state $\{\{1,2,\cdots,n\}\}$, called the most recent common ancestor (MRCA), and stays there forever. We set by convention that $\Pi^1(t)=\{\{1\}\}$ for any $t\geq 0.$ 

We introduce further notations: if $\pi$ is a partition of a set of integers, let $|\pi|$ be the number of blocks in $\pi;$ let $\Z_+=\{1,2,\cdots\}$ and $\mathcal P(\infty)$ be the set of partitions of $\Z_+$.

The Kingman's $n$-coalescents are consistent: for any $m>n\geq 1$, if we consider the natural restriction of $\Pi^{m}$ to the partitions in $\mathcal P(n)$, then the resulting new process %, that we denote by $\Pi^{m,n}$, 
has the same law as $\Pi^{n}$, thus independent of $m$. The consistency property will allow to construct the Kingman's (infinite) coalescent $\Pi^\infty=(\Pi^{\infty}(t))_{t\geq 0}$ which starts at time $0$ with the set of singletons $\{\{1\},\{2\},\cdots\}\in \mathcal P(\infty)$, and the restriction of $\Pi^\infty$ to $ 
\mathcal P(n)$ has the same law as $\Pi^n$ for all $n\geq 1$. This $\Pi^\infty$ can be constructed using Kolmogorov's extension theorem, see \cite[Proposition 2.1]{berestycki2009recent}. 

The process $\Pi^\infty$ can also be constructed naturally by first giving $\Pi^1$, and conditionally on $\Pi^n$ for $n\geq 1$, we use consistency property to construct $\Pi^{n+1}$ by connecting individual $n+1$ to $\Pi^{n}$ at a random time, see Figure \ref{fig:external}. More precisely, given $\Pi^{n}$: %, individual $n+1$ will be connected to $\Pi^n$ in the following way: 
at any time $t$, if individual $n+1$ has not been connected to $\Pi^n$, then the rate for it to be connected is equal to  $|\Pi^n(t)|$; if the connection takes place at time $t$, then the individual will coalesce with a block chosen uniformly from $\Pi^n(t)$. We denote the connection time of individual $n+1$ by $L_{n+1}$. We shall also call it the \textit{lineage length} of individual n+1. The construction just explained is the so-called {\it recursive (or sequential) construction} of Kingman’s coalescent; see \cite[Section 5]{dhersin2013length} for an introduction of recursive construction of $\Lambda$-coalescents for which Kingman's coalescent is a special case, see also \cite[Section 3.4]{crane2016ubiquitous}. 

We are interested in the asymptotic behaviour of the process of lineage lengths (or connection times) of individuals in this construction. This is partly motivated by a recent work (\citealp{favaro2019bayesian}) which studied the inference of quantities related to the genealogy of additional individuals given a known sample. The recursive construction provides a natural way of enlarging sample sizes, and could be a useful angle to investigate the genealogical relationship between known and new additional individuals. The idea of adding up small parts to construct the whole process can also be found in the measure division construction of $\Lambda$-coalescents (\citealp{yuanlambda}), see also \cite[Section 3]{berestycki2008small}.  %Note that there is a process called {\it the sequential Markov coalescent} \cite{mcvean2005approximating} which is different from the recursive construction here. 

%Back to the recursive construction, b
Based on the definition of recursive construction, for any $n\geq 1$, we have %if we denote the lineage length of individual $n+1$ by $L_{n+1}$ for $n\geq 1$, then 
\begin{equation}\label{eqn:lawln}\mathbb P(L_{n+1}\geq t\,|\,\Pi^n)=\exp\left(-\int_0^t|\Pi^n(s)|ds\right),\quad \forall t\geq 0.\end{equation}
Since $\Pi^1(t)=\{1\}$ for any $t\geq 0,$ we set $L_1=\infty$ by convention.  Note that in this construction, $L_{n+1}$ is the external branch length of  individual $n+1$ in $\Pi^{n+1}$. However as more individuals are added, the external branch length of individual $n+1$ will be shorter and shorter, see Figure \ref{fig:external}. From this point of view, we call $L_{n}$ the {\it provisional external branch length} (although for brevity we will still use \textit{lineage length} later) of individual $n$, for $n\geq 1.$ Here the case $n=1$ is included for completeness. 
\begin{figure}[!h]
\hspace*{1.3in}
\includegraphics[scale=0.6]{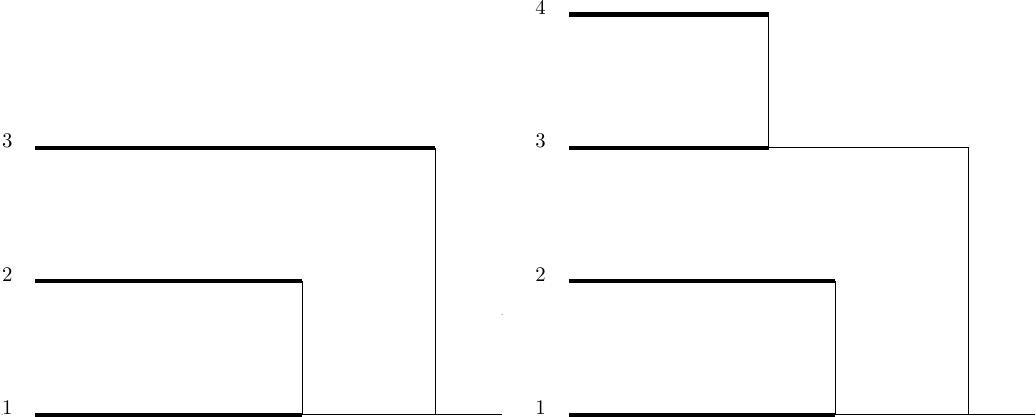}
\caption{On the left, the recursive construction is up to individual 3, and on the right is up to individual 4. The bold segments are the external branches. On the left, individual 3 is just added and thus $L_3$ is the external branch length of individual 3. On the right, since individual 4 coalesced with individual 3, the external branch length of individual 3 becomes $L_4$ which is smaller than $L_3$. }\label{fig:external}
\end{figure}
\begin{definition}
In the recursive construction of Kingman's coalescent, we call the process $$(L_n):=(L_n)_{n\geq 1}$$ 
{\it provisional external branch length sequence (PEBLS) of Kingman's coalescent}.
\end{definition}
Given $(L_n)$, we can construct $\Pi^\infty$, following the description of recursive construction: 
\begin{itemize}
\item for any $n\geq 2$, choose uniformly an element from $\{i: L_i\geq L_n, 1\leq i\leq n-1\}$, say $j$; 
\item then 
merge individual $n$ with the cluster containing $j$ at time $L_n.$ 
\end{itemize}
The resulting process has the same law as $\Pi^\infty.$

The recursive construction allows to build $\Pi^\infty$ by sample size expansion. We can view integer $n$ as individual $n$ and also as time $n$.  % main purpose of this paper is to study %$(L_n)$ as $n$ increases. %For convenience, we refer to $n$ as the $n$-th individual and also as the time $n$. 
%That is, we consider 
%the evolution of $L_n$ as $n$ increases (i.e., evolution by size expansion). 
This is the main difference with the usual Kingman's coalescent which fixes the sample size (finite or infinite) first and evolves in (real) time. A direct consequence of the size-expansion point of view  is that $(L_n)$ is not a Markov chain, since to determine the law of $L_n$, we need to know  not only $L_{n-1}$, but all $L_i$ for $2\leq i\leq n-1$, see \eqref{eqn:lawln}.  The main result of this paper is that, surprisingly, there is a Markov chain out of $(L_n)$, see Theorem \ref{thm:markov} in the next section. 

\section{Main results}\label{sec:main}

Let $M_j$ be the $j$-th largest  length in $(L_n)_{n\geq 2}$. Note that $M_1<\infty$ almost surely as Kingman's coalescent comes down from infinity (i.e.\ $|\Pi^\infty(t)|<\infty$ for all $t>0$, almost surely).  Let $A_1$ be the arrival time of $M_1: L_{A_1}=M_1$. Let $R_1=1.$ For any $i\geq 2,$ define $A_i, R_i$ by 
\begin{equation}\label{eqn:wr}A_i=\arg\max_{j}\{L_j: j>A_{i-1}\},\quad M_{R_i}=L_{A_i}.\end{equation}
In other words, 
\begin{itemize}
\item the largest length in $(L_n)_{n\geq 2}$ has index $A_1$; 
\item for any $i\geq 2$,  the largest among $\{L_n: n>A_{i-1}\}$ has index $A_{i}$; \item moreover, the length with index $A_i$ is the $R_i$-th largest among $(L_n)_{n\geq 2}$. 
\end{itemize}
Thus, $(A)=(A_i)_{i\geq 1}$ records the arrival times of successive largest lengths in $(L_n)_{n\geq 2}$, and $(R)=(R_i)_{i\geq 1}$ records the rankings of these lengths in $(L_n)_{n\geq 2}$. By definition we have 
\begin{equation}\label{eqn:ar>}1< A_i< A_{i+1}, 1\leq R_i< R_{i+1},\quad \text{ for any }i\geq 1.\end{equation}  
It turns out that $(R,A)$ is a Markov chain despite that $(L_n)$ is non-Markov. 
\begin{theorem}\label{thm:markov}
The process $(R,A)$ is a Markov chain such that 
\begin{itemize}
\item[(1)] $A_i-R_i\geq 1$, for any $i\geq 1$;  
\item[(2)] $\mathbb P(A_1=n)=\frac{2}{n(n+1)}$ for any $n\geq 2$, and $R_1=1$;
\item[(3)] For any $i\geq 1$ we have %$1\leq x\leq A_i-R_i$,  we have  
$R_i+1\leq R_{i+1} \leq A_i$ and for any $1\leq x\leq A_i-R_i$,
%The tail probabilities of the above two laws are given as follows:
\begin{equation}\label{eqn:rtail}\mathbb P(R_{i+1}\geq R_i+x\,|\, R_i,A_i)%=\frac{A_i+R_i+x}{A_i+R_i+1}\frac{{2A_i\choose A_i-R_i-x}}{{2A_i\choose A_i-R_i-1}}
=\frac{{2A_i-1\choose A_i-R_i-x}}{{2A_i-1\choose A_i-R_i-1}},\end{equation}
and for any $y\geq 1$, %  and 
\begin{equation}\label{eqn:atail}\mathbb P(A_{i+1}\geq A_i+y\,|\,R_i,A_i,R_{i+1})=\frac{A_i+R_{i+1}}{A_i+R_{i+1}+y-1}.\end{equation}
%%Consequently, 
%\begin{align}\label{eqn:mc}
%&\P(R_{i+1}=R_i+x, A_{i+1}=A_i+y\,|\,R_i,A_i)\nonumber\\
%=&\left(\prod_{k=1}^{x-1}\frac{A_i-R_i-k}{A_i+R_i+k}\right)\frac{2R_i+2x}{A_i+R_i+x}\prod_{k=1}^{y-1}\left(\frac{A_i+R_i+x+k-1}{A_i+R_i+x+k}\right)\frac{1}{A_i+R_1+x+y}\nonumber\\
%=&\frac{2R_i+2x}{(A_i+R_i+x+y-1)(A_i+R_i+x+y)}{A_i-R_i-1\choose A_i-R_i-x}\Big /{A_i+R_i+x-1\choose A_i+R_i}&
%\end{align}
\end{itemize}
\end{theorem}

To analyse the asymptotic behaviour of $(R,A)$, we study the convergence of the processes below $$\mathcal W^{(n)}:=\left(\left(\frac{R^2_{n+1+i}}{A_{n+1+i}}, \,\,\frac{A_{n+i}}{A_{n+1+i}}\right)\right)_{i\geq 1},\quad n\geq 0,$$ 
as $n\to\infty.$ Note that the above processes are not Markov for any $n$.  
%Let $Z$ be a positive random variable such that \begin{equation}\label{eqn:z}\P(Z\geq t)=e^{-t^2},\quad \forall t\geq 0.\end{equation} 
 % and any random variable $X\in\R_+:=[0,\infty)$,  let $X_{\geq t}\in[t,\infty)$ denote a new random variable with law given by 
%$$\mathcal L(X_{\geq t})=\mathcal L(X\, |\, X\geq t).$$
%A simple property is 
%\[(X_{\geq t})^2\stackrel{d}{=}(X^2)_{\geq t^2}.\]
To state the convergence result, we need some more notations. 
Let $\text{Exp(1)}$ be the exponential distribution with parameter $1$ and $\text{Uni}(0,1)$ the uniform distribution on $(0,1)$. Let $(\eta_1,\eta_2,\cdots)$ be i.i.d.\ random variables with common law $\text{Uni}(0,1)$. Let $\xi_0\geq 0$ be a random variable independent of $(\eta_1,\eta_2,\cdots)$.  Inductively, define 
\begin{equation}\label{eqn:xi}\xi_i:=(\xi_{i-1}+X_i)\eta_i, \quad i\geq 1\end{equation}
%where $Z^{(i)}$ is a random variable such that  $$Z^{(i)}=,$$ 
where $X_i\sim \text{Exp}(1)$ and is independent of $\{\xi_0, \xi_1,\cdots, \xi_{i-1}, \eta_1,\eta_2,\cdots\}$. Denote $\mathcal W:=((\xi_i, \eta_i))_{i\geq 1}$. We use $\Longrightarrow$ to denote the weak convergence in finite dimensional distributions. Then we have the following asymptotic result for $\mathcal W^{(n)}$.

\begin{theorem}\label{cor:limit}
If $\xi_0\sim \text{Exp}(1)$, then $\mathcal W^{(n)}\stackrel{n\to\infty}{\Longrightarrow} \mathcal W$ which is a stationary Markov chain with 
\begin{itemize}
\item $\xi_i\sim \text{Exp}(1)$ for any $i\geq 1$, 
\item and the density function of $(\xi_1,\eta_1)$ given by
\begin{equation}\label{density}
\frac{\ddr ^2}{\ddr s\ddr t}\mathbb P(\xi_1\leq s, \eta_1\leq t)=\frac{s}{t^2}e^{-\frac{s}{t}}, \quad s\geq 0, 0<t<1. 
\end{equation} 
\end{itemize}
%a stationary Markov chain $\mathcal W=((\xi_i, \eta_i))_{i\geq 1}$ such that  for any $i\geq 1,$ 
%$$\mathcal L(\eta_{i+1} \, |\, (\xi_i, \eta_i))=\text{Exp}(1). $$
%Moreover, conditioned on $((\xi_i, \eta_i), \eta_{i+1})$, we have  \begin{equation}\label{eqn:xii+1}\xi_{i+1}=Z_{\xi_i}^{(i+1)}e^{-\eta_{i+1}}, \end{equation} 
%where $Z_{\xi_i}^{(i+1)}$ is a random variable following (conditionally) the same distribution as $Z_{\xi_i}$ with $Z\sim \text{Exp}(1)$. Unconditionally, $\xi_i\sim \text{Exp}(1), \forall i\geq 1.$
\end{theorem}

The remaining part of the paper will be devoted to proofs. We will prove Theorem \ref{thm:markov} in Section \ref{sec:1}
 and Theorem \ref{cor:limit} in Section \ref{sec:2}.

\section{Proof of Theorem \ref{thm:markov}}\label{sec:1}

\subsection{Aldous's construction}
Aldous \cite[Section 4.2]{MR1673235} introduced a construction of Kingman's coalescent, see also \cite[Theorem 2.2]{berestycki2009recent}. The idea can be dated back to the seminal paper of Kingman (\citealp{kingman1982coalescent}). This construction will play a key role in the proof of Theorem \ref{thm:markov}. 

%\vspace{3mm}

\begin{figure}[!h]
 \vspace{-20pt}
\hspace*{0.85 in}
\includegraphics[scale=1.1]{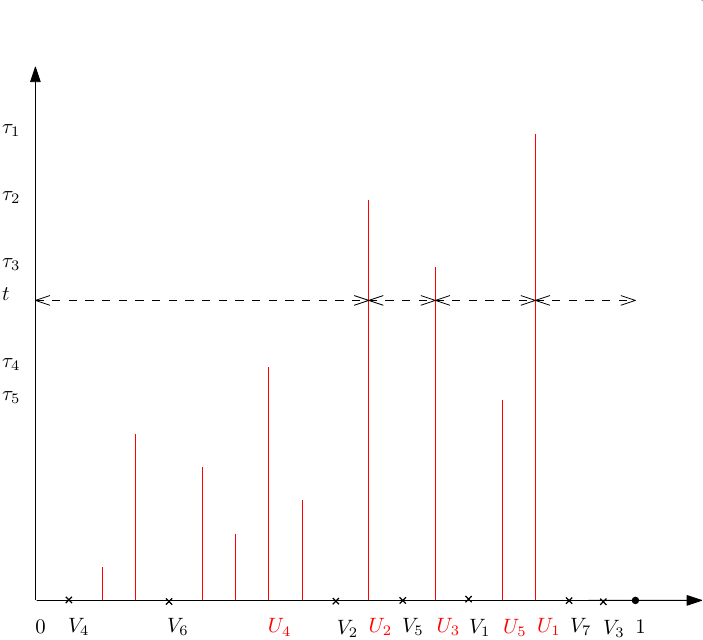}
\caption{Aldous's construction of Kingman's coalescent. The vertical axis is the time axis. The stick $i$ is at $U_i$ with height $\tau_i$. The $V_i$'s (the crosses) are the locations of individuals. For the 7 individuals in the figure, the partition at time $t$ is $\{ \{2,4,6\}, \{5\},\{1\},\{3,7\}\}.$}\label{fig:aldous}
\end{figure}

Let $(\zeta_i)_{i\geq 2}$ be independent random variables such that  $\zeta_i\sim \text{Exp}\left({i\choose 2}\right)$.  Define a random sequence $0<\cdots<\tau_3<\tau_2<\tau_1<\infty$ by $\tau_j=\sum_{k=j+1}^\infty \zeta_k.$ Let $(U_1,U_2,\cdots, V_1, V_2,\cdots)$ be i.i.d.\ 
random variables with common law $\text{Uni(0,1)}$ and independent of $(\tau_j)_{j\geq 1}$. Note that almost surely random variables in $U_i$'s and $V_i$'s are different from each other. Define a function $T:(0,1)\mapsto [0, \infty)$ such that $T(U_j)=\tau_j$ for any $j\geq 1$, and $T(u)=0$ if $u\notin\{U_1,U_2,\cdots\}$. We call the vertical line from $(U_i,\tau_i)$ down to $(U_i, 0)$ the stick $i$, see Figure \ref{fig:aldous}. Then for any $t\geq 0$, we define a partition of $\N$ such that $i,j$ are in the same block if and only if $$t\geq \sup_{\text{any } u\in(0,1) \text{ between } V_i,V_j}T(u).$$
The resulting process takes values in $\mathcal P(\infty)$ for any $t\geq 0$, and has the same law as $\Pi^\infty$. In this construction, we call $V_i$ the location of individual $i$, and $U_i$ the location of stick $i$. Note that this construction applies to finitely many individuals as a natural restriction, see an example of $7$ individuals in Figure \ref{fig:aldous}. 

Recall the sequence $(M_j)$ introduced at the beginning of  Section \ref{sec:main}. Using Aldous's construction, it is clear that $M_j=\tau_j$ for any $j\geq 1$. We shall from now on only use the notation $(\tau_j)$ as further discussions are based on Aldous's construction.

\subsection{Identify $(L_n)$ from Aldous's construction}\label{sec:identifyln}
Aldous's construction gives a realisation of $\Pi^\infty$, and thus we can use it to identify $(L_n)$ and $(R,A)$ in the recursive construction. We first deal with $(L_n)$. We start with this question: find $k$ such that $L_k=\tau_1$. Recall $U_1,U_2,\cdots, V_1,V_2,\cdots$ in Aldous's construction. 
\begin{lemma}\label{lem:tau_1}
%With individual $1$ at $V_1$ and stick $1$ at $U_1$, count individuals $2,3,\cdots$ at $V_2, V_3,\cdots$ one after another until 
Let $k$ be the smallest integer such that $U_1$ is between $V_k$ and $V_1$. Then $L_k=\tau_1.$
\end{lemma}
\begin{proof}
The stick $1$ at $U_1$ splits $(0,1)$ into two subintervals. WLOG, assume that $V_1,V_2,\cdots,V_{k-1}$ are on $(0,U_1)$ and $V_k$ is on $(U_1,1)$. By Aldous's construction, since the stick $1$ at $U_1$ is of length $\tau_1$ and is the largest among all sticks, all these $k$ individuals will merge into one cluster at time $\tau_1$, and the $k-1$ individuals on $(0,U_1)$ merge into one cluster at a time strictly smaller than $\tau_1$. Then the only possibility is that individual $k$ is connected to the coalescent process of the first $k-1$ individuals at time $\tau_1$. Then the lemma is proved. 
\end{proof}

Next we determine $L_n$ for any $n\geq 2$ (recall $L_1=\infty$). 
\begin{corollary}\label{cor:key}
%Given $n (\geq 2)$ individuals planted on $(0,1)$ at locations $V_1, V_2,\cdots, V_n$, 
%Count the sticks at locations $U_1,U_2,U_3,\cdots$ one after another until the smallest integer 
Let $n\geq 2$. Let $k$ be the smallest integer such that $V_n$ is single without any others from  $\{V_1,V_2,\cdots, V_n\}$ that is between $U_k$ and some $U_j$ for $1\leq j<k$ or between  $U_k$ and $0$ or between $U_k$ and $1$. Then $L_n=\tau_k$. 
\end{corollary}
\begin{proof}
The first $k-1$ sticks divide $(0,1)$ into $k$ subintervals.  Individual $n$ is located with some other individuals on one of the subintervals, say $(a,b)$, where $a,b$ are distinct elements in  $\{U_1, U_2,\cdots, U_{k-1}, 0,1\}$. The arrival of stick $k$ at $U_k$ will separate individual $n$ from others on $(a,b)$. This implies $L_n=\tau_k$, following a similar reasoning as in Lemma \ref{lem:tau_1}. Then the proof is finished. 
\end{proof}

We present three more corollaries which will be used for the identification of $(R,A)$. The first one finds the value of $n$ such that $L_n=\tau_k$ for $k\geq 1$, generalising Lemma \ref{lem:tau_1}. 
\begin{corollary}\label{cor:findn}
Let $k\geq 1$. Let $X$ be the closest element to $U_k$ from the left in $\{U_1,U_2,\cdots, U_{k-1},0,1\}$ and $Y$ from the right.  Let $n$ be the smallest integer such that the interval $(X,U_k)$  contains at least one element from $\{V_1,V_2,\cdots, V_n\}$ and the same for $(U_k,Y)$. Then $L_n=\tau_k$. 
\end{corollary}
\begin{proof}
For $k=1$, it is a restatement of Lemma \ref{lem:tau_1}. For $k\geq 2$, we apply Corollary \ref{cor:key} to obtain that for such a unique $n$ we have $L_n=\tau_k$. Then the proof is finished. 
\end{proof}

The second one presents a special scenario where an upper bound for $n$ in the above corollary can be given. 
\begin{corollary}\label{cor:xuy}
Let $s>1$ and consider $V_1,V_2,\cdots,V_s.$ Assume that for $k\geq 1$, $U_k$ is neighbour to some $V_i,V_j$ for $1\leq i\neq j\leq s$ (i.e.\ $U_k$ is between $V_i$ and $V_j$; there exists no other $U_a$ or $V_b$ between $V_i, V_j$ for $1\leq a\leq k-1, 1\leq b\leq s$). Then the individual $n$ with $L_n=\tau_k$ must have $n\leq s$.  
\end{corollary}
\begin{proof}
WLOG, assume $V_i<U_k<V_j$. Let $X$ be the closest element to $U_k$ from the left in  $\{U_1,U_2,\cdots,U_{k-1},0,1\}$  and $Y$ from the right. Then we have 
$$X<V_i<U_k<V_j<Y.$$
Then by Corollary \ref{cor:findn}, the individual $n$ such that $L_n=\tau_k$ must have $n\leq i\vee j\leq s.$ Then the proof is finished. 
\end{proof}

The last one finds a special scenario where a lower bound for $n$ in Corollary \ref{cor:findn} can be given.
\begin{corollary}\label{cor:s>m}
Let $m\geq 1, k\geq 1$. If there is no element from $\{V_1,V_2,\cdots,V_m\}$ that is located between $U_k$ and some $U_j$ for $1\leq j<k$, or between $U_k$ and $0$, or between $U_k$ and $1$, then the individual $n$ with $L_n=\tau_k$ must have $n>m$.  
\end{corollary}
\begin{proof}
To find $n$ such that $L_n=\tau_k$, we need to consider more $V's$ so that the condition in Corollary \ref{cor:findn} is satisfied. Therefore $n>m.$
\end{proof}
\subsection{Identify $(R,A)$ from Aldous's construction}\label{sec:identify}
Now we provide an algorithm to identify $(R,A)$ based on Corollary \ref{cor:findn}. 
\begin{corollary}\label{cor:shell}
We have $R_1=1$ by definition and the value of $A_1$ is given by Lemma \ref{lem:tau_1} or Corollary \ref{cor:findn}: $A_1$ is the smallest integer such  that $U_1$ is between $V_{A_1}$ and $V_1$. We have $L_{A_1}=\tau_1$. 

In general, given $(R_i, A_i)$ for some $i\geq 1$, we perform the following loop to obtain $(R_{i+1}, A_{i+1})$. 
\begin{itemize}
\item for $j\geq R_i+1$
\item find $n$ using Corollary \ref{cor:findn} such  that $L_n=\tau_{j}$. If $n<A_i,$ then continue the loop with $j=j+1$; otherwise (i.e.\ $n>A_i$) let $A_{i+1}=n$ and $R_{i+1}=j$ (hence $L_{A_{i+1}}=\tau_{R_{i+1}}$), and get out of the loop.
\end{itemize}
\end{corollary}
\begin{proof}
We only need to check how the algorithm produces $(R_{i+1}, A_{i+1})$ given $(R_i, A_i)$.  We count $R_i+1, R_i+2, R_i+3,\cdots$ until the first integer $k$ such that the lineage length of rank $k$ belongs to an individual, say $m$, which is larger than $A_i$. Then we have found $R_{i+1}=k, A_{i+1}=m,$ based on the definition of $(R,A)$. This is a restatement of the algorithm in the corollary and thus the proof is finished. \end{proof}

It is clear that Corollary \ref{cor:shell} is like a shell with the core being Corollary \ref{cor:findn}. We will provide another way of identifying $(R,A)$  so that in Corollary \ref{cor:shell} we do not need to find the exact value of $n$ to know that $n<A_i$ (thanks to Corollary \ref{cor:xuy}) and there is a natural way of determining $n$ if $n>A_i$ (thanks to Corollary \ref{cor:s>m}). The proof of Theorem \ref{thm:markov} will reply on this new approach that we introduce in the next section. 

\subsection{A simpler way of identifying $(R,A)$ from Aldous's construction}\label{sec:simple}

%In Aldous's construction, the $U_i$'s and $V_i$'s are given at the beginning.  we will plant the sticks and individuals successively in the order of $U_1, U_2, \cdots$ and $V_1, V_2,\cdots$ respectively. More precisely, we 
Note that Aldous's construction has two systems of notation: $U_i$'s and $V_i$'s, and also sticks and individuals. Stick $i$ is at $U_i$ and individual $i$ is at $V_i$. Once $U_i$'s and $V_i$'s are given, the sticks and individuals are automatically planted at their locations. To facilitate the description of the simpler way of identification, we will plant the sticks and individuals successively in the order of $U_1,U_2,\cdots$ and $V_1,V_2,\cdots$ respectively. The implementation of the  identification can be decomposed into two steps.

\vspace{2mm}

\noindent Step 1: plant the stick $1$ at $U_1$. Plant individuals $1,2,3,\cdots$ successively on $(0,1)$ at locations $V_1,V_2, V_3, \cdots$
until the first integer $n$ ($n\geq 2$) such that $U_1$ is between $V_n$ and $V_1$. We know from Lemma \ref{lem:tau_1} that $A_1=n$ and $R_1=1$.

\vspace{2mm}

\noindent Step 2: we proceed by induction to show the transition from $(R_i,A_i)$ to $(R_{i+1}, A_{i+1})$ for any $i\geq 1$.  Assume the following is true: \begin{itemize}
\item We know $R_i$ already, and  after planting individual, say $h$, we obtain $A_i=h$; 
\item  $U_1, U_2, \cdots, U_{R_i}$ and $V_1, V_2, \cdots, V_{A_i}$ are planted locations; 
\item Moreover, every $U_l$ for $1\leq l\leq R_i$ is neighbour to some $V_k$ and $V_j$ for $1\leq k\neq j\leq A_i$ (i.e.\ $U_l$ is between $V_k,V_j$; there exists no other $U_s$ or $V_t$ between $V_k,V_j$ for $1\leq s\leq R_i, 1\leq t\leq A_i$). 
\end{itemize}
Note that the notion of neighbour here is based on the planted locations. The above assumptions hold for $i=1$, see Step 1.  To obtain $(R_{i+1}, A_{i+1})$, we do the following. 

\begin{itemize}
\item Finding $R_{i+1}$.  Plant remaining sticks successively until the first stick $n$ at $U_n$,  which is not neighbour to some $V_k$ and $V_j$ for $1\leq k\neq j\leq A_i$. Then we have found $R_{i+1}=n.$ Note that in this case, there are three possibilities (and only one of them can happen): 
\begin{enumerate}
\item $U_n$ is neighbour to some  $U_l$ and  $V_j$ for $1\leq l\leq n-1, 1\leq j\leq A_i$;
\item $U_n$ is neighbour to $0$ and some  $V_j$ for $1\leq j\leq A_i$;
\item $U_n$ is neighbour to $1$ and some  $V_j$ for $1\leq j\leq A_i$.
\end{enumerate}
\item Finding $A_{i+1}.$ Next we plant the remaining individuals successively until the first individual $m$ such that $V_m$ is neighbour to $U_n$ and $U_l$ or $U_n$ and $0$ or $U_n$ and $1$ depending on which one of the three possibilities happened in the step above finding $R_{i+1}$.  Then we have found $A_{i+1}=m$.  Moreover, the assumptions made at the beginning of Step 2 hold for $(R_{i+1}, A_{i+1}).$ 
\end{itemize}
\begin{proposition}
The above procedure identifies $(R,A).$
\end{proposition}
\begin{proof}
Starting from $(R_i,A_i)$, before finding $R_{i+1}$, we are in the scenario described in Corollary \ref{cor:xuy}. That is, for any $R_i+1\leq n<R_{i+1}$, we have $L_s=\tau_n$ for some $s\leq A_i$. The searching of $R_{i+1}$ stops at $n$ when we are in the scenario described in Corollary \ref{cor:s>m} as more individuals need to be planted successively until finding the individual (which is $A_{i+1}$) that has the lineage length $\tau_{n}$ (we use Corollary \ref{cor:findn} to see that this individual has length $\tau_n$) and thus $R_{i+1}=n$. The proof is finished. 
\end{proof}

\subsection{Proof of Theorem \ref{thm:markov}}

\begin{proof}[Proof of Theorem \ref{thm:markov} - (1)] In the identification procedure described in Section \ref{sec:simple} for identifying $(R,A)$, at any step $i$, we have planted exactly $A_i$ individuals and $R_i$ sticks, in such a way that each stick has to be neighbour to two individuals. Therefore necessarily $A_i\geq R_i+1$.   
%First of all $R_1=1 $ and $A_1\geq 2$, then $A_1-R_1\geq 1$. For any $i\geq 2,$ recall that $L_{A_i}$ is the largest among all lengths that arrive later than $A_{i-1}$ in the recursive construction, and is the $R_i$-th largest among all in $(L_n)_{n\geq 2}$.  If $\min_{2\leq j\leq A_{i-1}}L_j> L_{A_i}$, then $\{L_2,\cdots, L_{A_{i-1}}\}$ constitute the $A_{i-1}-1$ largest lengths in $(L_n)_{n\geq 2}$. That implies $R_{i}=A_{i-1}$. If $L_{A_i}>\min_{2\leq j\leq A_{i-1}}L_{j}$, then $R_i<A_{i-1}$. Therefore in either case we have $R_i\leq A_{i-1}$. Moreover by \eqref{eqn:ar>}, we have $A_i\geq A_{i-1}+1$. Then we conclude that $A_i-R_i\geq 1$ for any $i\geq 1$, which means statement (1) holds true.  
\end{proof}

%For the rest of the proof, we use a new way of looking at Aldous's construction to characterise $(R,A)$ and prove the remaining statements (2)-(4).  The identification procedure in Section \ref{sec:identify} still applies, but

\begin{proof}[Proof of Theorem \ref{thm:markov} - (2)]To determine the law of $A_1$, we first recall the following lemma which is well known, see for instance \cite{dev86}. 
\begin{lemma}\label{lem:simplex}Assume there are $k (k\geq 1)$ i.i.d. uniform random variables on $(0,1)$. Then $(0,1)$ is cut into $k+1$ subintervals whose lengths are exchangeable and the vector of lengths follows the uniform distribution on a standard $k$-simplex. If we plant another independent uniform random variable on $(0,1)$, then it will enter one of the $k+1$ subintervals with equal probability.  Conditioned on entering any subinterval, the resulting $k+2$ subintervals are again exchangeable and the new vector follows the uniform distribution on a standard $(k+1)$-simplex. \end{lemma}
Now we recall that $(U_1,U_2,\cdots, V_1, V_2,\cdots)$ are i.i.d.\ uniform on $(0,1)$. Then using the above lemma and Lemma \ref{lem:tau_1}, we have for any $n\geq 2$, 
\begin{align*}
\mathbb P(A_1=n)&=2\mathbb P(\max_{1\leq i\leq n-1}V_i <U_1,  V_n>U_1)\\
&=2\mathbb P(V_1<U_1)\left(\prod_{j=1}^{n-2}\mathbb P(V_{j+1}<U_1\,|\,\max_{1\leq i\leq j}V_i <U_1)\right)\mathbb P(V_{n}>U_1\,|\,\max_{1\leq i\leq n-1}V_i <U_1)\\
&=2\times\frac{1}{2}\times\prod_{j=1}^{n-2} \frac{j+1}{j+2}\times\frac{1}{n+1}=\frac{2}{n(n+1)}. \end{align*}
Then statement (2) is proved. 
\end{proof}

\begin{proof}[Proof of Theorem \ref{thm:markov} - (3)] 
We will show the probability mass functions for the two laws to deduce the tail probabilities. More precisely, we shall first prove the following: 
\begin{itemize}
\item For any $i\geq 1$ we have $R_i+1\leq R_{i+1}\leq A_i$ and for any $1\leq x\leq A_i-R_i,$
\begin{equation}\label{eqn:ri+1}
\begin{split}\mathbb P(R_{i+1}=R_i+x\,|\,R_i,A_i)
&=\frac{2R_i+2x}{A_i+R_i+x}\left(\prod_{k=1}^{x-1}\frac{A_i-R_i-k}{A_i+R_i+k}\right)\\
&=\frac{2R_i+2x}{A_i+R_i+1}{2A_i\choose A_i-R_i-x}\Big/{2A_i\choose A_i-R_i-1},
\end{split}
\end{equation}
\item and for any $y\geq 1$,  
\begin{align}\label{eqn:ai+1}
\begin{split}
\mathbb P(A_{i+1}=A_i+y\,|\,R_i,A_i,R_{i+1})
&=\frac{1}{A_i+R_{i+1}+y}\prod_{k=1}^{y-1}\left(\frac{A_i+R_{i+1}+k-1}{A_i+R_{i+1}+k}\right)\\
&=\frac{A_i+R_{i+1}}{(A_i+R_{i+1}+y-1)(A_i+R_{i+1}+y)}.
\end{split}
\end{align}
\end{itemize}

We will use the identification procedure in Section \ref{sec:simple}. Under the assumptions in Step 2, $(0,1)$ is divided into $R_i+A_i+1$ subintervals. Among them, there are three categories of subintervals: 
\begin{enumerate}
\item there are $R_i$ pairs of subintervals such that each pair share the same $U_k$ as a common end for some $1\leq k\leq R_i$; 
\item there are two subintervals such that either of them has one end being $0$ or $1$ (cannot have both $0,1$ as ends); 
\item the remaining $A_i-R_i-1$ subintervals can only have ends from $\{V_j: 1\leq j\leq A_i\}$. 
\end{enumerate}

Then following Step 2, we plant remaining sticks starting from $R_i+1$ successively until the first stick $R_i+x$ that is not neighbour to some $V_k$ and $V_j$ for $1\leq k\neq j\leq A_i$. Note that every stick $j$ for $R_i+1\leq j<R_i+x$ enters a subinterval of category 3, and thus killing one subinterval of category 3 and adding a pair of subintervals of category 1. The stick $R_i+x$ will enter a subinterval of category 1 or 2. In other words, 
\begin{align*}\mathbb P(R_{i+1}=R_i+x\,|\, R_i,A_i)&=\mathbb P(\text{stick $j$ enters a subinterval of category 3 for $R_i+1\leq j<R_i+x$},\\
& \quad \quad \text{ and stick $R_i+x$ enters a subinterval of category 1 or 2}\,|\,R_i,A_i)\end{align*}
Then clearly we have $R_i+1\leq R_{i+1}\leq A_i$ since the number of subintervals of category 3 is $A_i-R_i-1$. Using Lemma \ref{lem:simplex} and conditional probability formula,  the above display yields the first equality in \eqref{eqn:ri+1}. The second equality is a direct simplification. 

The next step is to plant remaining individuals until stick $R_i+x$ is again neighbour to two planted individuals. We omit the proof of \eqref{eqn:ai+1}, which is very similar to that of \eqref{eqn:ri+1}. 

Next we deduce the tail probabilities. We will only show \eqref{eqn:rtail} as \eqref{eqn:atail} is straightforward. If $x=A_i-R_i$, then we obtain 
$$\mathbb P(R_{i+1}\geq A_i\,|\, R_i,A_i)=\frac{1}{{2A_i-1\choose A_i-R_i-1}}=\mathbb P(R_{i+1}=A_i\,|\, R_i,A_i).$$ For $1\leq x<A_i-R_i$, we show that \eqref{eqn:rtail} implies \eqref{eqn:ri+1}: 
\begin{align*}
&\mathbb P(R_{i+1}\geq R_i+x\,|\, R_i,A_i)-\mathbb P(R_{i+1}\geq R_i+x+1\,|\, R_i,A_i)\\
=&\frac{{2A_i-1\choose A_i-R_i-x}}{{2A_i-1\choose A_i-R_i-1}}-\frac{{2A_i-1\choose A_i-R_i-x-1}}{{2A_i-1\choose A_i-R_i-1}}\\
=&\frac{A_i+R_i+x}{A_i+R_i+1}\frac{{2A_i\choose A_i-R_i-x}}{{2A_i\choose A_i-R_i-1}}
-\frac{A_i-R_i-x}{A_i+R_i+1}\frac{{2A_i\choose A_i-R_i-x}}{{2A_i\choose A_i-R_i-1}}\\
=&\frac{2R_i+2x}{A_i+R_i+1}\frac{{2A_i\choose A_i-R_i-x}}{{2A_i\choose A_i-R_i-1}}
&
\end{align*}
which is exactly the probability $\mathbb P(R_{i+1}= R_i+x\,|\, R_i,A_i)$ given by \eqref{eqn:ri+1}. Then we conclude that \eqref{eqn:rtail} holds true. 
\end{proof}

Finally, all statements in Theorem \ref{thm:markov} are proved, and the proof is complete. 

\section{Proof of Theorem \ref{cor:limit}}\label{sec:2}

\subsection{Preliminaries}
In this section, we prove two lemmas for preparation. We use $\xrightarrow[]{w}$ to denote the weak convergence of probability measures; $\xrightarrow[]{d}$ for the convergence in distribution for random variables and $\stackrel{d}{=}$ for being equal in distribution. We use $\mathcal L(\cdot)$ to denote the law of a random object. 

\begin{lemma}\label{lem:forr}
Let $n\geq 1.$ Consider a random variable $W:=W_n$ such that 
$$\mathbb P(W=k)=ck{2n\choose n-k},\quad  0\leq k\leq n,$$
where $c>0$ is the normalising constant. Let $t>0.$ Then uniformly in $s\in[0,t]$, we have 
\begin{equation}\label{eqn:local}\sqrt{n}\mathbb P(W=\lfloor s\sqrt{n} \rfloor )\longrightarrow 2se^{-s^2}, \text{ as } n\to\infty.\end{equation}
As a consequence, 
$$\mathcal L\left(\frac{W^2}{n}\right)\xrightarrow[n\to\infty]{w} \text{Exp}(1).$$
\end{lemma}
\begin{proof} It suffices to prove \eqref{eqn:local}. We first find the asymptotic equivalent of $c$. Let $\alpha\sim B(2n,1/2)$, a binomial random variable with parameters $2n$ and $1/2$. Note that 
$$\frac{1}{c}=\sum_{k=0}^nk{2n\choose n-k}=\sum_{k=0}^nn{2n\choose n-k}-\sum_{k=0}^{n}(n-k){2n\choose n-k}=:I_1-I_2.$$
For the first term, we have 
$$2^{-2n}I_1=n\mathbb P(\alpha\leq n)=\frac{n}{2}+\frac{n}{2}\mathbb P(\alpha=n)=\frac{n}{2}+\frac{n}{2}2^{-2n}{2n\choose n}=\frac{n}{2}+\frac{\sqrt{n}}{2\sqrt{\pi}}(1+o(1)),$$
where the last equality is due to Stirling formula.  For the second term, we have 
\begin{align*}
2^{-2n}I_2&=\sum_{k=0}^nk{2n\choose k}2^{-2n}=2n\sum_{k=1}^{n}{2n-1\choose k-1}2^{-2n}=2n\sum_{k=0}^{n-1}{2n-1\choose k}2^{-2n}\\
&=2n\frac{\sum_{k=0}^{2n-1}{2n-1\choose k}}{2}2^{-2n}=2n\times\frac{2^{2n-1}}{2}2^{-2n}=\frac{n}{2}.
\end{align*}
Then we obtain that 
\begin{equation}\label{eqn:1/c}\frac{1}{c}= \frac{\sqrt{n}}{2\sqrt{\pi}}2^{2n}(1+o(1)).\end{equation}
Therefore,  uniformly for $s\in[0,t],$ as $n\to\infty$, 
$$\mathbb P(W=\lfloor s\sqrt{n} \rfloor )=\frac{2\sqrt{\pi}}{\sqrt{n}}2^{-2n}\lfloor s\sqrt{n} \rfloor {2n\choose n-\lfloor s\sqrt{n} \rfloor }(1+o(1))=e^{-s^2}\frac{2s}{\sqrt{n}}(1+o(1)),$$
where the first $o(1)$ comes from \eqref{eqn:1/c} and the second equality follows from Stirling formula, and both $o(1)$'s converge to $0$ uniformly in $s\in[0,t]$ as $n\to\infty$. Then the proof is finished.
\end{proof}

\begin{lemma}\label{lem:tight}
The process $\left(\frac{R_i^2}{A_i}\right)_{i\geq 1}$ is tight. 
\end{lemma}
\begin{proof}
Tightness means that for any $\epsilon>0,$ there exists a compact set $E\subset[0,\infty)$ such that for any $i\geq 1,$ $\mathbb P\left(\frac{R_i^2}{A_i}\notin E\right)\leq \epsilon.$ To prove this, it suffices to show that there exists $c>0$ such that the following holds true
\begin{equation}\label{eqn:bounded}
\E\left[\frac{R_i}{\sqrt{A_i}}\right]\leq c, \quad \forall i\geq 1.
\end{equation}
 Indeed, applying Markov inequality, we obtain that $\mathbb P\left(\frac{R_i^2}{A_i}> N\right)=\mathbb P\left(\frac{R_i}{\sqrt{A_i}}> \sqrt{N}\right)\leq \frac{c}{\sqrt{N}}\leq \epsilon$ if we take $N$ large. Then the compact set can be set as $E=[0,N]$ and the tightness is obtained. 

Denote $I:=\frac{\sum_{x=1}^{A_i-R_i-1}{2A_i-1\choose A_i-R_i-x-1}}{{2A_i-1\choose A_i-R_i-1}}$. Note that using \eqref{eqn:rtail}, we have 
\begin{align*}
&\E[A_i-R_{i+1}\, |\, R_i, A_i]\\
&=\sum_{x=1}^{A_i-R_i-1}(A_i-R_i-x)\left(\frac{{2A_i-1\choose A_i-R_i-x}}{{2A_i-1\choose A_i-R_i-1}}-\frac{{2A_i-1\choose A_i-R_i-x-1}}{{2A_i-1\choose A_i-R_i-1}}\right)\\
&=\sum_{x=1}^{A_i-R_i-1}(A_i-R_i-x)\frac{{2A_i-1\choose A_i-R_i-x}}{{2A_i-1\choose A_i-R_i-1}}-\sum_{x=1}^{A_i-R_i-1}(A_i-R_i-x-1+1)\frac{{2A_i-1\choose A_i-R_i-x-1}}{{2A_i-1\choose A_i-R_i-1}}\\
&=\sum_{x=1}^{A_i-R_i-1}(2A_i-1)\frac{{2A_i-2\choose A_i-R_i-x-1}}{{2A_i-1\choose A_i-R_i-1}}-\sum_{x=1}^{A_i-R_i-2}(2A_i-1)\frac{{2A_i-2\choose A_i-R_i-x-2}}{{2A_i-1\choose A_i-R_i-1}}-I\\
&=(2A_i-1)\sum_{x=1}^{A_i-R_i-2}\frac{{2A_i-2\choose A_i-R_i-x-1}-{2A_i-2\choose A_i-R_i-x-2}}{{2A_i-1\choose A_i-R_i-1}}+\frac{2A_i-1}{{2A_i-1\choose A_i-R_i-1}}-I\\
&=(2A_i-1)\frac{{2A_i-2\choose A_i-R_i-2}}{{2A_i-1\choose A_i-R_i-1}}-I\\
&=A_i-R_i-1-I.
\end{align*}Then
\begin{equation}\label{eqn:Rdecom}\E[R_{i+1}\,|\,R_i,A_i]=R_i+1+I.\end{equation}
Note that 
\begin{align}\label{eqn:I<}
I&\leq \frac{\sum_{x=1}^{A_i-R_i}{2A_i\choose A_i-R_i-x}}{{2A_i\choose A_i-R_i}}&\\
&\leq \frac{1}{2R_i+2}\frac{\sum_{x=1}^{A_i-R_i}2(R_i+x){2A_i\choose A_i-R_i-x}}{{2A_i\choose A_i-R_i}}=\frac{A_i+R_i+1}{2R_i+2}\frac{{2A_i\choose A_i-R_i-1}}{{2A_i\choose A_i-R_i}}=\frac{A_i-R_i}{2R_i+2}\nonumber\\
&\label{eqn:I1}\leq \frac{A_i}{R_i},
\end{align}
where the first equality is due to \eqref{eqn:ri+1}. 
Conditional on $A_i$ and $R_i$, let $\alpha\sim B(2A_i, 1/2)$. If $R_i\leq \sqrt{A_i}$, then using \eqref{eqn:I<}, we have \begin{equation}
\begin{split}
I&\leq 2^{2A_i}\frac{\sum_{x=1}^{A_i-R_i}{2A_i\choose A_i-R_i-x}2^{-2A_i}}{{2A_i\choose A_i-\lfloor \sqrt{A_i}\rfloor}}=2^{2A_i}\frac{\mathbb P(0\leq \alpha < A_i-R_i\, |\, A_i,R_i)}{{2A_i\choose A_i-\lfloor \sqrt{A_i}\rfloor}}\\
&\leq \frac{2^{2A_i}}{{2A_i\choose A_i-\lfloor \sqrt{A_i}\rfloor}}\leq c_0 \sqrt{A_i}, \quad \text{ if $i$ (or equivalently $A_i$, see \eqref{eqn:ar>}) is large enough,}
\end{split}
\end{equation}
where  $c_0>0$ does not depend on anything and the last inequality is by Stirling formula.  In conclusion, for $i$ large enough, 
\begin{equation}\label{eqn:2bounds}
\E[R_{i+1}\,|\,R_i,A_i]\leq \begin{cases}R_i+1+\frac{A_i}{R_i}, \quad \text{ if }R_i>  \sqrt{A_i}; \\
R_i+1+c_0 \sqrt{A_i},\quad \text{ if }R_i\leq \sqrt{A_i}.\end{cases}
\end{equation}
Next we note that there exists $0<c_1<1$ such that for $i\geq 1$
\begin{align}\label{eqn:denominator}
\E\left[\frac{1}{\sqrt{A_{i+1}}}\,|\,R_i, A_i,R_{i+1}\right]%\leq \sum_{y=1}^\infty \frac{A_i}{\sqrt{A_i+y}(A_i+y-1)(A_i+y)}
&\leq \mathbb P(A_{i+1}\geq 2A_i \,|\, R_i,A_i,R_{i+1})\frac{1}{\sqrt{2A_i}}+\mathbb P(A_{i+1}< 2A_i \,|\, R_i,A_i,R_{i+1})\frac{1}{\sqrt{A_i}}\nonumber\\
&\leq \left(\frac{1}{2\sqrt{2}}+\frac{1}{2}\right)\frac{1}{\sqrt{A_i}}\nonumber\\ 
&\leq \frac{c_1}{\sqrt{A_i}}, \end{align}
where the first inequality uses $A_{i+1}>A_i$, see \eqref{eqn:ar>}, and the second inequality uses $\mathbb P(A_{i+1}\geq 2A_i\,|\, R_i,A_i,R_{i+1})\geq \frac{1}{2}$ which is a consequence of \eqref{eqn:atail}.  Then following \eqref{eqn:Rdecom}, \eqref{eqn:2bounds} and \eqref{eqn:denominator}, we obtain for $i$ large enough,
\begin{align*}
\E\left[\frac{R_{i+1}}{\sqrt{A_{i+1}}}\,|\,R_i,A_i\right]&=\E\left[\E\left[\frac{R_{i+1}}{\sqrt{A_{i+1}}}\,|\,R_i,A_i,R_{i+1}\right]\,|\,R_i,A_i\right]\\
&\leq c_1\E\left[\frac{R_{i+1}}{\sqrt{A_i}}\,|\,R_i, A_i\right]\\
&\leq \begin{cases}
c_1\frac{R_i}{\sqrt{A_i}}+c_1\frac{1}{\sqrt{A_i}}+c_1\frac{\sqrt{A_i}}{R_i},\quad \text{ if }R_i>\sqrt{A_i}\\
c_1\frac{R_i}{\sqrt{A_i}}+c_1\frac{1}{\sqrt{A_i}}+c_0c_1,\quad \text{ if }R_i\leq  \sqrt{A_i}\\
\end{cases}\\
&\leq \begin{cases}
c_1\frac{R_i}{\sqrt{A_i}}+2c_1,\quad \text{ if }R_i> \sqrt{A_i}\\
c_1\frac{R_i}{\sqrt{A_i}}+c_1(1+c_0),\quad \text{ if }R_i\leq  \sqrt{A_i}\\
\end{cases}\\
&\leq c_1\frac{R_i}{\sqrt{A_i}}+c_2
\end{align*}
where $c_2=\max\{2c_1, c_1(1+c_0)\}$, and for the third inequality we used \eqref{eqn:ar>}. Therefore, we obtain the following 
\begin{equation}\label{eqn:tightness}
\E\left[\frac{R_{i+1}}{\sqrt{A_{i+1}}}\right]\leq c_1\E\left[\frac{R_{i}}{\sqrt{A_{i}}}\right]+c_2, \quad \text{if $i$ is large enough.}
\end{equation}
Since $0<c_1<1$, the above display yields \eqref{eqn:bounded}. Then the proof is finished. 
\end{proof}

\subsection{Proof of Theorem \ref{cor:limit}}
We start with two propositions. Either of them proves a partial result of Theorem \ref{cor:limit} whose proof is provided at the end of this section. The first proposition is about the ergodicity of the Markov chain $(\xi_i)_{i\geq 0}$, introduced in \eqref{eqn:xi}. Let $P$ be the transition kernel of $(\xi_i)_{i\geq 0}.$ Introduce the following weighted supremum norm for a function $f:[0, \infty)\mapsto \R$
$$\|f\|=\sup_{x\geq 0}\frac{|f(x)|}{x+1}.$$
Let $C_b$ be the set of bounded continuous functions from $[0,\infty)$ to $\R$. For $t\geq 0$, let $\mathcal E_t$ be the law as follows:
\begin{equation}\label{eqn:et}
\mathcal E_t: = \mathcal L(X \, |\, X\geq t)=\mathcal L(t+X), \quad \text{ for } X\sim \text{Exp}(1), t\geq 0,
\end{equation}
In particular, $\mathcal E_0=\text{Exp}(1)$.

\begin{proposition}[Geometric ergodicity]
\label{prop1}
The Markov chain $(\xi_i)_{i\geq 0}$ admits a unique invariant measure $\mu=\text{Exp}(1)$. Moreover, there exists $C>0$ and $\rho\in(0,1)$ such that 
\begin{equation}\label{eqn:unique}\|P^nf-\mu(f)\|\leq C\rho^n\|f-\mu(f)\|,\quad \text{ for any } f\in C_b,\end{equation}
where $\mu(f)=\int_0^\infty f(x)\mu(\ddr x)$.
\end{proposition}
\begin{proof}
First of all, we verify that $\text{Exp}(1)$ is an invariant measure. Thanks to the construction \eqref{eqn:xi}, it suffices to show that, if $\xi_0\sim \text{Exp}(1)$, then $\xi_1\sim \text{Exp}(1)$. Recall $\xi_1=(\xi_{0}+X_1)\eta_1$ where  $\xi_0, X_1, \eta_1$ are independent and $X_1\sim \text{Exp}(1), \eta_1\sim \text{Uni(0,1)}$ (see \eqref{eqn:xi}). If we assume $\xi_0\sim \text{Exp}(1)$, then for any integer $k\geq 0$, we have 
\begin{align*}
\E[\xi_1^k]&=\E[(\xi_{0}+X_1)^k]\E[\eta_1^k]\\
&=\frac{1}{k+1}\E[(\xi_0+X_1)^k]\\
&=\frac{1}{k+1}\sum_{i=0}^k{k\choose i}\E[\xi_0^i]\E[X_1^{k-i}]=k!=\E[\xi_0^k].
\end{align*}
%Here we used the fact that $\xi_0$ and $X_1$ are i.i.d.\ with law Exp(1). 
Since the above display is true for any $k$, we conclude that $\xi_1\stackrel{d}{=}\xi_0\sim \text{Exp}(1)$. Here we used the method of moments, see \cite[Theorem 30.1]{billingsley1995probability}. Thus $\text{Exp}(1)$ is indeed an invariant measure of $(\xi_i)_{i\geq 0}$.

Note that \eqref{eqn:unique} implies that $\mu=\text{Exp}(1)$ is the unique invariant measure. For the rest, we will prove \eqref{eqn:unique} by applying Theorem 3.6 in \cite{hairer2010convergence} in our context. It suffices to verify the following condition which combines Assumption 3.1 and Remark 3.5 in \cite{hairer2010convergence}. 

\textit{Condition:  
(1) Let $V$ be the identity function: $V(x)=x, \forall x\geq 0$. Then $PV(x)=\frac{V(x)+1}{2}$ for any $x\geq 0.$  (2) Let $R>0.$ Then  for any $f$ with $\sup_{x\geq 0}|f(x)|\leq 1$, we have 
$$|Pf(x)-Pf(y)|\leq 2(1-e^{-R}), \quad\,\forall\,  0\leq x\leq y\leq  R.$$}

The verification of the above condition is as follows. Note that 
$$PV(x)=\E[V(\xi_1)\,|\,\xi_0=x]=\E[\xi_{0}+X_1]\E[\eta_1]=\frac{x+1}{2}=\frac{V(x)+1}{2}.$$
Then the first statement is proved.  For the second statement, recall that $\xi_0, X_1,\eta_1$ are independent. Then we observe that
\begin{align*}
Pf(x)&=\E[f(\xi_1 \,|\, \xi_0=x)]\\
&=\E[f((\xi_{0}+X_1)\eta_1)\,|\,\xi_0=x]\\
&=\E[f((\xi_{0}+X_1)\eta_1){\bold 1}_{X_1\geq y-x}\,|\,\xi_0=x]+\E[f((\xi_{0}+X_1)\eta_1){\bold 1}_{X_1< y-x}\,|\,\xi_0=x]\\
&=\mathbb P(X_1\geq y-x\,|\,\xi_0=x)\E[f((\xi_0+X_1)\eta_1)\,|\,X_1\geq y-x, \xi_0=x]\\
&\quad \quad+\E[f((\xi_{0}+X_1)\eta_1){\bold 1}_{X_1< y-x}\,|\,\xi_0=x]\\
&=\mathbb P(X_1\geq y-x)Pf(y)+\E[f((\xi_{0}+X_1)\eta_1){\bold 1}_{X_1< y-x}\,|\,\xi_0=x].\end{align*}
Here for the last equality we used that $\mathcal L(X_1\,|\,X_1\geq y-x)=\mathcal L(X_1+y-x)$ since $X_1\sim \text{Exp}(1)$, see \eqref{eqn:et}. Then since $|f(\cdot)|\leq 1$ and $0\leq x\leq y\leq R$, we obtain 
\begin{align*}
|Pf(x)-Pf(y)|&\leq 2\mathbb P(X_1<y-x)\leq 2(1-e^{-R}). \end{align*}
%where the equality uses the independence between $X_1$ and $\xi_0$. 
Thus the condition holds true and the proof is finished. 
\end{proof}

The second proposition is to prove the one dimensional convergence of $\left(\frac{R_i^2}{A_i}\right)$.   For preparation, the following lemma is needed. 
\begin{lemma}\label{lem:a/a}
For any $c\geq 0$, we have  $\mathcal L\left(\frac{A_{i}}{A_{i+1}}\,|\, R_i, R_{i+1}=\lfloor  cA_i\rfloor \right)=\mathcal L\left(\frac{A_{i}}{A_{i+1}}\,|\,  R_{i+1}=\lfloor  cA_i\rfloor\right)$. Moreover,  for any $C>0, f\in C_b,$
\[\sup_{0\leq c\leq C}\left|\E\left[f\left(\frac{A_{i}}{A_{i+1}}\right)\,|\, R_{i+1}=\lfloor \sqrt{cA_i}\rfloor\right]-\E[f(\eta)]\right|\xrightarrow[]{i\to\infty}0, \]
where $\eta\sim \text{Uni}(0,1)$. 
\end{lemma}
\begin{proof}Note that here we allow $c=0$ which means $R_i=0$. For Kingman's coalescent, $R_i$ can never be $0,$ but the transition probabilities \eqref{eqn:ri+1} and \eqref{eqn:ai+1} do allow the more general case with $R_i=0.$ 

The lemma is a direct consequence of \eqref{eqn:atail} or \eqref{eqn:ai+1}  and the fact that $A_i\to\infty$ as $i\to\infty$.  We omit details of the proof. 
\end{proof}
\begin{proposition}
\label{prop2}
The law of $\frac{R_i^2}{A_i}$ converges weakly to \text{Exp}(1) as $i\to\infty$. 
\end{proposition}
\begin{proof}
We claim that it suffices to prove that for any $t\geq 0$ and any $f\in C_b$, 
\begin{equation}\label{eqn:uniform}
\sup_{0\leq s\leq t}\left|\E\left[f\left(\frac{R^2_{i+1}}{A_{i+1}}\right)\,|\,R_i=\lfloor \sqrt{s A_i}\rfloor\right]-Pf(s)\right|\xrightarrow[]{i\to\infty}0.\end{equation}
Indeed, if \eqref{eqn:uniform} is true, then we have that for any $\epsilon>0$ there exists $T>0$ such that 
\begin{equation}\label{eqn:T} \inf_{0\leq s\leq t}\mathbb P\left(0\leq \frac{R_{i+1}^2}{A_{i+1}}\leq T\,|\, R_i=\lfloor \sqrt{sA_i}\rfloor \right)\geq 1-\epsilon,\quad i \text{ large enough}.\end{equation}
%Note that for $s$ in the interval $[0,t]$,  the convergence in \eqref{eqn:uniform} is uniform, and $Pf(s)$ is continuous and thus uniformly continuous. Therefore, 
Then as $i\to\infty$, uniformly for $s\in[0,t]$,  we have \begin{align*} \E\left[f\left(\frac{R^2_{i+2}}{A_{i+2}}\right)\,|\,R_i=\lfloor \sqrt{s A_i}\rfloor\right]&= \E\left[\E\left[f\left(\frac{R^2_{i+2}}{A_{i+2}}\right)\,|\, \frac{R_{i+1}^2}{A_{i+1}} \right] \,|\, R_i=\lfloor \sqrt{s A_i}\rfloor\right]\\
%&\approx \E\left[Pf(\bar s) \,|\, R_i=\lfloor \sqrt{s A_i}\rfloor\right] 
&= \E\left[Pf\left(\frac{R_{i+1}^2}{A_{i+1}}\right) \,|\, R_i=\lfloor \sqrt{s A_i}\rfloor\right] +o(1)
\\
&= P^2f(s)+o(1),
\end{align*}
where the first equivalence is due to \eqref{eqn:T} and \eqref{eqn:uniform} and $f\in C_b, Pf\in C_b;$ the second equivalence is due to \eqref{eqn:uniform} and $Pf\in C_b.$ Therefore we obtain
\[\sup_{0\leq s\leq t}\left|\E\left[f\left(\frac{R^2_{i+2}}{A_{i+2}}\right)\,|\,R_i=\lfloor \sqrt{s A_i}\rfloor\right]-P^2f(s)\right|\xrightarrow[]{i\to\infty}0.\]
Iterating this procedure, it follows that for any $n\geq 1,$
\[\sup_{0\leq s\leq t}\left|\E\left[f\left(\frac{R^2_{i+n}}{A_{i+n}}\right)\,|\,R_i=\lfloor \sqrt{s A_i}\rfloor\right]-P^nf(s)\right|\xrightarrow[]{i\to\infty}0,\]
which implies that $\mathcal L(\frac{R^2_{i+n}}{A_{i+n}}\,|\,R_i=\lfloor \sqrt{s A_i}\rfloor)$ converges weakly to $\mathcal L(\xi_n\,|\, \xi_0=s
)$ for any $s\geq 0$ as $i\to\infty$.   By  Proposition \ref{prop1}, $\mathcal L(\xi_n\,|\, \xi_0=s
)$ converges weakly to \text{Exp}(1) as $n\to\infty$. Finally we apply Lemma \ref{lem:tight} to conclude  that this proposition holds true. This reasoning also leads to $\left(\frac{R_i^2}{A_i},\frac{R_{i+1}^2}{A_{i+1}}\cdots, \frac{R_{i+k}^2}{A_{i+k}}\right)\xrightarrow[i\to\infty]{d}(\xi_0,\xi_1,\cdots,\xi_k)$ for any positive integer $k$ if $\xi_0\sim\text{Exp}(1)$. But we are content with the one dimensional convergence for this proposition and the multidimensional convergence will be proved in the proof of Theorem \ref{cor:limit} more straightforwardly.  Another way to prove this proposition using \eqref{eqn:uniform} is to apply \cite[Theorem (1)]{karr1975weak}. The only problem is that $\left(\frac{R_i^2}{A_i}\right)$ is not a Markov chain. It suffices to enlarge it into  $\left(\frac{R_i^2}{A_i}, \frac{1}{A_i}\right)$. We omit the detailed steps.

To prove \eqref{eqn:uniform}, we first recall the random variable $W=W_n$ in Lemma \ref{lem:forr}. For $0\leq  k\leq n$, \eqref{eqn:ri+1} yields
$$\mathcal L(W\, | \, W> k)=\mathcal L(R_{i+1}\,|\, R_i=k, A_i=n).$$
Then using \eqref{eqn:local}, we obtain 
\begin{equation}\label{eqn:p1}\sup_{0\leq s\leq t}\left|\E\left[f\left(\frac{R^2_{i+1}}{A_{i}}\right)\,|\,R_i=\lfloor \sqrt{s A_i}\rfloor\right]-\E[f(Z_s)]\right|\xrightarrow[]{i\to\infty}0, \end{equation}
where $Z_s\sim \mathcal E_s$ (see \eqref{eqn:et}).
As a consequence, similar to \eqref{eqn:T},  for any $\epsilon>0$, there exists $C>0$ such that 
\begin{equation}\label{eqn:r<a}\inf_{0\leq s\leq t}\mathbb P(0\leq R_{i+1}^2\leq C A_i\,|\, R_i=\lfloor \sqrt{s A_i}\rfloor)\geq 1-\epsilon,\quad i \text{ large enough}.\end{equation}
Then using Lemma \ref{lem:a/a}, we have
\begin{equation}\label{eqn:p2}\sup_{0
\leq s\leq t, s\leq c\leq C}\left|\E\left[f\left(\frac{A_{i}}{A_{i+1}}\right)\,|\,R_i=\lfloor \sqrt{s A_i}\rfloor,  R_{i+1}=\lfloor \sqrt{cA_i}\rfloor\right]-\E[f(\eta)]\right|\xrightarrow[]{i\to\infty}0. \end{equation}
where $\eta\sim\text{Uni}(0,1)$.  Note that $Pf(s)=\E[f(Z_s\eta)]$ if we assume $\eta$ is independent of $Z_s$. Moreover, 
$\frac{R^2_{i+1}}{A_{i+1}}=\frac{R^2_{i+1}}{A_{i}}\frac{A_i}{A_{i+1}}.$
Then using the above three displays, we conclude that \eqref{eqn:uniform} holds true and thus the proof for the proposition is finished.
\end{proof}

\begin{proof}[Proof of Theorem \ref{cor:limit}]
By Proposition \ref{prop1}, $(\xi_i)_{i\geq 0}$ is a stationary Markov chain if $\xi_0\sim \text{Exp}(1).$ Using the definition \eqref{eqn:xi}, we conclude that $\mathcal W=(\xi_i, \eta_i)_{i\geq 1}$ is a stationary Markov chain. Next we show \eqref{density}. Using \eqref{eqn:xi}, we have 
\begin{align*}\mathbb P(\xi_1\leq s, \eta_1\leq t )&=\mathbb P((\xi_0+X_1) \eta_1\leq s, \eta_1\leq t )=\int_0^t  \gamma\left(2,\frac{s}{u}\right)\ddr u,&
\end{align*}
where $\gamma(\cdot,\cdot)$ is the lower incomplete gamma function. Here we used that $\xi_0, X_1, \eta_1$ are independent and $\xi_0\stackrel{d}{=}X_1\sim \text{Exp(1)}$ and $\eta_1\sim \text{Uni}(0,1)$, and such that $\xi_0+X_1$ follows the Gamma distribution with shape parameter $2$ and scale parameter $1$. Taking the partial derivative with respect to $s$ and $t$ yields \eqref{density}.  It remains to prove $\mathcal W^{(n)}\Longrightarrow \mathcal W$ as $n\to\infty.$

Let $\xi_0\sim \text{Exp}(1)$. Using Proposition \ref{prop2} and \eqref{eqn:p1}, we have 
\begin{equation}\label{eqn:2}\left(\frac{R^2_{n}}{A_n},\frac{R^2_{n+1}}{A_n}\right)\xrightarrow[n\to\infty]{d} (\xi_0, \xi_{0}+X_1).\end{equation}
Together with Lemma \ref{lem:a/a} and the fact that $\left(\frac{R_i^2}{A_i}\right)_{i\geq 1}$ is tight, we obtain further that 
\begin{equation}\label{eqn:ind}\left(\frac{R^2_{n}}{A_n},\frac{R^2_{n+1}}{A_n}, \frac{A_{n}}{A_{n+1}},  \frac{A_{n+1}}{A_{n+2}},\cdots\right)\stackrel{n\to\infty}{\Longrightarrow} (\xi_0, \xi_{0}+X_1, \eta_1, \eta_2,\cdots).\end{equation}
Note that 
$$\frac{R^2_{i+1}}{A_{i+1}}=\frac{R^2_{i+1}}{A_{i}}\frac{A_i}{A_{i+1}},\quad \forall i\geq 1.$$
The above two displays entail 
$$\left(\frac{R^2_{n}}{A_n},\frac{R^2_{n+1}}{A_n},\frac{R^2_{n+1}}{A_{n+1}}, \frac{A_{n}}{A_{n+1}},  \frac{A_{n+1}}{A_{n+2}}, \cdots\right)\stackrel{n\to\infty}{\Longrightarrow} (\xi_0, \xi_{0}+X_1,\xi_1, \eta_1, \eta_2,\cdots).$$
Since the asymptotic behaviour of $\frac{R^2_{i+1}}{A_{i}}$ only depends on that of $\frac{R^2_{i}}{A_{i}}$ (see \eqref{eqn:2} and \eqref{eqn:ind}), we have
$$\left(\frac{R^2_{n}}{A_n},\frac{R^2_{n+1}}{A_n},\frac{R^2_{n+1}}{A_{n+1}}, \frac{R^2_{n+2}}{A_{n+1}},\frac{A_{n}}{A_{n+1}},  \frac{A_{n+1}}{A_{n+2}},\cdots\right)\stackrel{n\to\infty}{\Longrightarrow} (\xi_0, \xi_{0}+X_1,\xi_1, \xi_{1}+X_2, \eta_1, \eta_2,\cdots). $$
%, \quad \text{ as $n\to\infty$}.$$
Iterating this procedure, we obtain that for any $k\geq 0,$
\begin{align*}&\left(\frac{R^2_{n}}{A_n},\frac{R^2_{n+1}}{A_n},\cdots,\frac{R^2_{n+k}}{A_{n+k}},\frac{R^2_{n+k+1}}{A_{n+k}}, \frac{A_{n}}{A_{n+1}},  \frac{A_{n+1}}{A_{n+2}},\cdots\right)\\
&\stackrel{n\to\infty}{\Longrightarrow} (\xi_0, \xi_{0}+X_1,\cdots, \xi_k, \xi_{k}+X_{k+1},  \eta_1, \eta_2,\cdots). %, \quad \text{ as $n\to\infty$}
\end{align*}
which implies $\mathcal W^{(n)}\stackrel{n\to\infty}{\Longrightarrow} \mathcal W$. Then the proof of Theorem \ref{cor:limit} is finished. 
\end{proof}

\section*{Acknowledgements}
The author thanks Matthias Birkner for suggesting the name {\it provisional external branch length sequence} for $(L_n)$, and Cl\'ement Foucart and Huili Liu for helpful comments on a draft version. The author thanks an
anonymous referee for carefully reading the manuscript and for pointing out several mistakes and giving insightful and constructive comments. 
\bibliographystyle{alea3}
\bibliography{myref}

\end{document}